\documentclass[a4paper,oneside]{amsart}

\usepackage{amssymb,amsthm}
\usepackage{hyperref}

\newtheorem{thm}{Theorem}[section]
\newtheorem{pro}[thm]{Proposition}
\newtheorem{lem}[thm]{Lemma}

\newtheorem{exa}[thm]{Example}

\theoremstyle{definition}

\newtheorem*{acknow}{Acknowledgments}

\theoremstyle{remark}

\newcommand{\en}{\mathbb{N}}
\newcommand{\zet}{\mathbb{Z}}
\newcommand{\er}{\mathbb{R}}

\newcommand{\ce}{\mathbb{C}}

\newcommand{\sign}{\mathrm{sign}\,}
\newcommand{\me}{\mathrm{e}}

\begin{document}

\title[Coarse and uniform embeddings]{Coarse and uniform embeddings between Orlicz sequence spaces}

\author{Michal Kraus}

\address{Department of Mathematical Analysis, Faculty of Mathematics and
Physics, Charles University, Sokolovsk\'a~83, 186~75 Praha~8,
Czech Republic}

\address{Universit\'{e} de Franche-Comt\'{e}, Laboratoire de Math\'{e}matiques UMR 6623, 16 route de Gray, 25030 Besan\c{c}on Cedex, France}

\email{mkraus@karlin.mff.cuni.cz}



\begin{abstract}
We give an almost complete description of the coarse and uniform embeddability between Orlicz sequence spaces. We show that the embeddability between two Orlicz sequence spaces is in most cases determined only by the values of their upper Matuszewska-Orlicz indices. On the other hand, we present examples which show that sometimes the embeddability is not determined by the values of these indices.
\end{abstract}

\thanks{This research was partially supported by the grants GA\v{C}R 201/11/0345 and PHC Barrande 2012-26516YG}


\maketitle

\section{Introduction}

Let $(M,d_M),(N,d_N)$ be metric spaces and let $f:M\to N$ be a mapping. Then $f$ is called a \emph{coarse embedding} if there exist nondecreasing functions $\rho_1,\rho_2:[0,\infty)\to[0,\infty)$ such that $\lim_{t\to\infty}\rho_1(t)=\infty$ and $$\rho_1(d_M(x,y))\leq d_N(f(x),f(y))\leq\rho_2(d_M(x,y))\quad\text{for all }x,y\in M.$$ We say that $f$ is a \emph{uniform embedding} if $f$ is injective and both $f$ and $f^{-1}:f(M)\to M$ are uniformly continuous. Following Kalton \cite{kalton07} we call $f$ a \emph{strong uniform embedding} if $f$ is both a coarse embedding and a uniform embedding. Naturally we say that $M$ \emph{coarsely embeds} into $N$ if there exists a coarse embedding of $M$ into $N$, and similarly for other types of embeddings. Let us mention that what we call a coarse embedding is called a uniform embedding by some authors. We use the term coarse embedding because in the nonlinear geometry of Banach spaces the term uniform embedding has a well established meaning as above.

The study of conditions under which a Banach space coarsely (or uniformly) embeds into another Banach space has been a very active area of the nonlinear geometry of Banach spaces. Coarse embeddability has received much attention in recent years mainly because of its connection with geometric group theory, whereas the study of uniform embeddability may be regarded as classical. See \cite{kalton08} for a recent survey on the nonlinear geometry of Banach spaces.

Not much is known in general, but there are some partial results. The coarse and uniform embeddability between $\ell_p$-spaces is now completely cha\-ra\-cte\-ri\-zed. Let us recall the results. Nowak proved that $\ell_p$ coarsely embeds into $\ell_2$ if $1\leq p<2$ \cite[Proposition 4.1]{nowak05} and that $\ell_2$ coarsely embeds into $\ell_p$ for any $1\leq p<\infty$ \cite[Corollary 4]{nowak06}. A construction due to Albiac in \cite[proof of Proposition 4.1(ii)]{albiac}, originally used to show that $\ell_p$ Lipschitz embeds into $\ell_q$ if $0<p<q\leq1$, can be used to show that $\ell_p$ strongly uniformly embeds into $\ell_q$ if $1\leq p<q$ (see also \cite{albiacbaudier}, where this construction is performed for all $0<p<q$). This fact also follows from Proposition \ref{emb_orlicz_lp} below, whose proof is based on Albiac's construction. On the other hand, Johnson and Randrianarivony proved that $\ell_p$ does not coarsely embed into $\ell_2$ if $p>2$ \cite[Theorem 1]{johnsonrandria}. Later, results of Mendel and Naor \cite[Theorems~1.9 and 1.11]{mendelnaor} showed that $\ell_p$ actually does not coarsely or uniformly embed into $\ell_q$ if $p>2$ and $q<p$. Furthermore, $\ell_2$ uniformly embeds into $\ell_p$ if $1\leq p<\infty$. Indeed, by \cite[Corollary 8.11]{benylind}, $\ell_2$ uniformly embeds into $S_{\ell_2}$, which is uniformly homeomorphic to $S_{\ell_p}$ by \cite[Theorem 9.1]{benylind}. In fact, $\ell_2$ even strongly uniformly embeds into $\ell_p$ if $1\leq p<2$. This will be proved in Theorem~\ref{l2_emb} below. We can summarize the results as follows.

\begin{thm}\label{lp_embeddings}
Let $p,q\in[1,\infty)$. Then the following assertions are equivalent:
\renewcommand{\labelenumi}{\rm(\roman{enumi})}
\begin{enumerate}
\item $\ell_p$ coarsely embeds into $\ell_q$.
\item $\ell_p$ uniformly embeds into $\ell_q$.
\item $\ell_p$ strongly uniformly embeds into $\ell_q$.
\item $p\leq q$ or $q<p\leq2$.
\end{enumerate}
\end{thm}

Our aim is to generalize this classification to a wider class of Banach spaces, namely to Orlicz sequence spaces. Let $h_M$ and $h_N$ be Orlicz sequence spaces associated with Orlicz functions $M$ and $N$, and let $\beta_M$ and $\beta_N$ be the upper Matuszewska-Orlicz indices of the functions $M$ and $N$. We will show that the coarse (uniform) embeddability of $h_M$ into $h_N$ is in most cases determined only by the values of $\beta_M$ and $\beta_N$. The dependence of the embeddability of $h_M$ into $h_N$ on the values of $\beta_M$ and $\beta_N$ is very similar to the dependence of the embeddability of $\ell_p$ into $\ell_q$ on the values of $p$ and $q$ from Theorem \ref{lp_embeddings} (note that the upper Matuszewska-Orlicz index of $\ell_p$ is $p$). In some cases, however, the embeddability of $h_M$ into $h_N$ is not determined by the values of $\beta_M$ and $\beta_N$. A brief summary of our results is given at the end of the paper.

It is worth mentioning that Borel-Mathurin proved in \cite{borel1} the following result concerning uniform homeomorphisms (i.e. bijections which are uniformly continuous and their inverses are also uniformly continuous) between Orlicz sequence spaces. Let $M$ and $N$ be Orlicz functions and let $\alpha_M$ and $\alpha_N$ be their lower Matuszewska-Orlicz indices. If $h_M$ and $h_N$ are uniformly homeomorphic, then $\alpha_M=\alpha_N$ and $\beta_M=\beta_N$. The fact that $\alpha_M=\alpha_N$ was published also in \cite{borel2}, the fact that $\beta_M=\beta_N$ is a consequence of results of Kalton \cite{kalton_unif_homeo}.

This paper is organized as follows. In Section \ref{preliminaries} we summarize the notation and terminology, and recall basic facts concerning Orlicz sequence spaces. In Section \ref{embeddings_of_l2} we give the proof of the fact that $\ell_2$ strongly uniformly embeds into $\ell_p$ if $1\leq p<2$. Section \ref{main_results} contains the results concerning the coarse and uniform embeddability between Orlicz sequence spaces.

\section{Preliminaries}\label{preliminaries}

Our notation and terminology for Banach spaces is standard, as may be found for example in \cite{lindtzI} and \cite{lindtzII}. All Banach spaces throughout the paper are supposed to be real. The unit sphere of a Banach space $X$ is denoted by $S_X$. If $(X_n)_{n=1}^\infty$ is a sequence of Banach spaces and $1\leq p<\infty$, then $\left(\sum_{n=1}^\infty X_n\right)_{\ell_p}$ stands for the $\ell_p$-sum of these spaces, i.e. the space of all sequences $x=(x_n)_{n=1}^\infty$ such that $x_n\in X_n$ for every $n$, and $\|x\|=\left(\sum_{n=1}^\infty \|x_n\|^p\right)^\frac{1}{p}<\infty$. If a Banach space $X$ is isomorphic to a subspace of a Banach space $Y$, we will sometimes say that $X$ \emph{linearly embeds into}~$Y$.
 
Let us give the necessary background concerning Orlicz sequence spaces. Details may be found in \cite{lindtzI} and \cite{lindtzII}.

A function $M:[0,\infty)\to[0,\infty)$ is called an \emph{Orlicz function} if it is continuous, nondecreasing and convex, and satisfies $M(0)=0$ and $\lim_{t\to\infty}M(t)=\infty$. 

Let $M$ be an Orlicz function. We denote by $\ell_M$ the Banach space of all real sequences $(x_n)_{n=1}^\infty$ satisfying $\sum_{n=1}^\infty M\left(\frac{|x_n|}{\rho}\right)<\infty$ for some $\rho>0$, equipped with the norm defined for $x=(x_n)_{n=1}^\infty\in\ell_M$ by $$\|x\|=\inf\left\{\rho>0:\sum_{n=1}^\infty M\left(\frac{|x_n|}{\rho}\right)\leq1\right\}.$$ Let $h_M$ denote the closed subspace of $\ell_M$ consisting of all $(x_n)_{n=1}^\infty\in\ell_M$ such that $\sum_{n=1}^\infty M\left(\frac{|x_n|}{\rho}\right)<\infty$ for every $\rho>0$. The sequence $(e_n)_{n=1}^\infty$ of canonical vectors then forms a symmetric basis of $h_M$. Clearly if $M(t)=t^p$ for some $1\leq p<\infty$, then $h_M$ is just the space $\ell_p$ with its usual norm.  

If $M(t)=0$ for some $t>0$, then $M$ is said to be \emph{degenerate}. In this case, $h_M$ is isomorphic to $c_0$ and $\ell_M$ is isomorphic to $\ell_\infty$. In the sequel, \emph{Orlicz functions are always supposed to be nondegenerate}. 

We will be interested in the spaces $h_M$. Note that $\ell_M=h_M$ if and only if $\ell_M$ is separable if and only if $\beta_M<\infty$, where $\beta_M$ is defined below.

An important observation is that if two Orlicz functions $M_1$ and $M_2$ coincide on some neighbourhood of 0, then $h_{M_1}$ and $h_{M_2}$ consist of the same sequences and the norms induced by $M_1$ and $M_2$ are equivalent.

The lower and upper Matuszewska-Orlicz indices of $M$ are defined by
$$\alpha_M=\sup\left\{q\in\er:\sup_{\lambda,t\in(0,1]}\frac{M(\lambda t)}{M(\lambda)t^q}<\infty\right\},$$
$$\beta_M=\inf\left\{q\in\er:\inf_{\lambda,t\in(0,1]}\frac{M(\lambda t)}{M(\lambda)t^q}>0\right\},$$ respectively. Then $1\leq\alpha_M\leq\beta_M\leq\infty$. Note also that if $M(t)=t^p$ for some $1\leq p<\infty$, then $\alpha_M=\beta_M=p$. We will need the following theorem due to Lindenstrauss and Tzafriri (see \cite[Theorem 4.a.9]{lindtzI}).

\begin{thm}\label{lt_embeddings}
Let $M$ be an Orlicz function and let $1\leq p\leq\infty$. Then $\ell_p$ if $p<\infty$, or $c_0$ if $p=\infty$, is isomorphic to a subspace of $h_M$ if and only if $p\in[\alpha_M,\beta_M]$.
\end{thm}

Let $M$ be an Orlicz function and $x=(x_n)_{n=1}^\infty\in h_M$. Using Lebesgue's dominated convergence theorem we see that the function $$\rho\mapsto\sum_{n=1}^\infty M\left(\frac{|x_n|}{\rho}\right),\ \rho>0,$$ is continuous. In particular, 
\begin{equation}\label{sum_eq1}
\sum_{n=1}^\infty M\left(\frac{|x_n|}{\|x\|}\right)=1.
\end{equation}

The following lemma is a simple consequence of the convexity of $M$ combined with the fact that $M(0)=0$, and \eqref{sum_eq1}.

\begin{lem}\label{sum_vs_norm}
Let $M$ be an Orlicz function and let $x=(x_n)_{n=1}^\infty\in h_M$.
\renewcommand{\labelenumi}{\rm(\alph{enumi})}
\begin{enumerate}
\item If $\|x\|\leq1$, then $\sum_{n=1}^\infty M(|x_n|)\leq\|x\|$.
\item If $\|x\|\geq1$, then $\sum_{n=1}^\infty M(|x_n|)\geq\|x\|$.
\end{enumerate}
\end{lem}

If $X$ is a Banach space, define $q_X=\inf\left\{q\geq2:X\text{ has cotype }q\right\}$. Then if $M$ is an Orlicz function, we have \begin{equation}\label{cotype_orlicz}q_{h_M}=\max(2,\beta_M).\end{equation} This can be proved as follows. Suppose first that $\beta_M<\infty$. Note that $h_M$, equipped with the natural order, is a Banach lattice. By Remark 2 after Proposition 2.b.5 in \cite{lindtzII}, we have $$\beta_M=\inf\left\{1<q<\infty:h_M\text{ satisfies a lower $q$-estimate}\right\}.$$ By \cite[Theorem~1.f.7]{lindtzII}, if a Banach lattice satisfies a lower $r$-estimate for some $1<r<\infty$, then it is $q$-concave for every $r<q<\infty$. And by \cite[Proposition~1.f.3(i)]{lindtzII}, if a Banach lattice is $q$-concave for some $q\geq2$, then it is of cotype $q$. Hence $q_{h_M}\leq\max(2,\beta_M)$. The opposite inequality follows from the fact that $\ell_{\beta_M}$ is isomorphic to a subspace of $h_M$ by Theorem \ref{lt_embeddings}, and $q_{\ell_{\beta_M}}=\max(2,\beta_M)$. If $\beta_M=\infty$, then, by Theorem \ref{lt_embeddings}, $h_M$ contains $c_0$, and the result follows.

\section{\texorpdfstring{Embeddings of $\ell_2$}{Embeddings of l2}}\label{embeddings_of_l2}

In this section we give the promised proof of the fact that $\ell_2$ strongly uniformly embeds into $\ell_p$ if $1\leq p<2$. The proof is inspired by Nowak's construction of coarse embeddings between these spaces in \cite[proof of Corollary~4]{nowak06}.

Recall that a kernel $K$ on a set $X$ (i.e. a function $K:X\times X\to\ce$ such that $K(y,x)=\overline{K(x,y)}$ for all $x,y\in X$) is called 
\begin{enumerate}\renewcommand{\labelenumi}{\rm(\alph{enumi})}
\item \emph{positive definite} if $\sum_{i,j=1}^nK(x_i,x_j)c_i\overline{c_j}\geq0$ for all $n\in\en$, $x_1,\dots,x_n\in X$ and $c_1,\dots,c_n\in\ce$,
\item \emph{negative definite} if $\sum_{i,j=1}^nK(x_i,x_j)c_i\overline{c_j}\leq0$ for all $n\in\en$, $x_1,\dots,x_n\in X$ and $c_1,\dots,c_n\in\ce$ satisfying $\sum_{i=1}^n c_i=0$. 
\end{enumerate}
Note that if the kernel $K$ is real-valued, then in order to check the positive or negative definiteness of $K$ it suffices to use only the real scalars.
 
Recall also that for $p,q\in[1,\infty)$, the Mazur map $M_{p,q}:S_{\ell_p}\to S_{\ell_q}$, defined for $x=(x_n)_{n=1}^\infty$ by $$M_{p,q}(x)=\left(|x_n|^{\frac{p}{q}}\sign{x_n}\right)_{n=1}^\infty,$$ is a uniform homeomorphism between these unit spheres. If $p>q$, then there exists $C>0$ such that for all $x,y\in S_{\ell_p}$ we have the inequalities 
\begin{equation}\label{mazur_estim}
C\|x-y\|^\frac{p}{q}\leq\|M_{p,q}(x)-M_{p,q}(y)\|\leq\frac{p}{q}\|x-y\|,
\end{equation}
and the opposite inequalities if $p<q$ (with different $C$) because clearly $M_{q,p}=M_{p,q}^{-1}$. See \cite[Theorem 9.1]{benylind} for a proof. 

\begin{thm}\label{l2_emb}
Let $1\leq p<2$. Then $\ell_2$ strongly uniformly embeds into $\ell_p$.
\end{thm}
\begin{proof}
First, for every $t>0$ there exists a mapping $\varphi_t:\ell_2\to S_{\ell_2}$ such that for all $x,y\in\ell_2$ we have
\begin{equation}\label{dadague_map}
\|\varphi_t(x)-\varphi_t(y)\|^2=2\left(1-e^{-t\|x-y\|^2}\right).
\end{equation}
To prove this statement, fix $t>0$. By a simple computation, the function $(x,y)\mapsto\|x-y\|^2$, $(x,y)\in\ell_2\times\ell_2$, is a negative definite kernel on $\ell_2$, and therefore, by \cite[Proposition 8.4]{benylind}, the function $(x,y)\mapsto\me^{-t\|x-y\|^2}$, $(x,y)\in\ell_2\times\ell_2$, is a positive definite kernel on $\ell_2$. By \cite[Proposition 8.5(i)]{benylind}, there exists a Hilbert space $H$ and a mapping $T:\ell_2\to H$ such that $\me^{-t\|x-y\|^2}=\left\langle T(x),T(y)\right\rangle$ for all $x,y\in\ell_2$. Then $$\|T(x)-T(y)\|^2=2\left(1-e^{-t\|x-y\|^2}\right)$$ for all $x,y\in\ell_2$, and we may suppose that $H$ is real, infinite-dimensional and separable (since $T$ is continuous), i.e. $H=\ell_2$. Take $\varphi_t=T$. Clearly $\|\varphi_t(x)\|=1$ for every $x\in\ell_2$.

Let $t_n>0$, $n\in\en$, be such that $\sum_{n=1}^\infty\sqrt{t_n}<\infty$. For each $n\in\en$, define $f_n=M_{2,p}\circ\varphi_{t_n}$. Let $x_0\in\ell_2$ be arbitrary and define $f:\ell_2\to\left(\sum_{n=1}^\infty\ell_p\right)_{\ell_p}$ by $f(x)=(f_n(x)-f_n(x_0))_{n=1}^\infty$ (that $f(x)\in\left(\sum_{n=1}^\infty\ell_p\right)_{\ell_p}$ for every $x\in\ell_2$ will follow from the estimate \eqref{l2_emb_upper} below). Let us show that $f$ is a strong uniform embedding. Since the spaces $\left(\sum_{n=1}^\infty\ell_p\right)_{\ell_p}$ and $\ell_p$ are isometric, the proof will then be complete.

Let $x,y\in\ell_2$. Then
\begin{align}\label{l2_emb_upper}
\|f(x)-f(y)\|&\leq\sum_{n=1}^\infty\|f_n(x)-f_n(y)\|=\sum_{n=1}^\infty\|M_{2,p}(\varphi_{t_n}(x))-M_{2,p}(\varphi_{t_n}(y))\|\\
\nonumber&\leq\frac{2}{p}\sum_{n=1}^\infty\|\varphi_{t_n}(x)-\varphi_{t_n}(y)\|=\frac{2\sqrt{2}}{p}\sum_{n=1}^\infty\left(1-e^{-t_n\|x-y\|^2}\right)^\frac{1}{2}\\
\nonumber&\leq\frac{2\sqrt{2}}{p}\sum_{n=1}^\infty\left(t_n\|x-y\|^2\right)^\frac{1}{2}
=\left(\frac{2\sqrt{2}}{p}\sum_{n=1}^\infty\sqrt{t_n}\right)\|x-y\|,
\end{align}
where the first inequality follows from the triangle inequality, the second inequality from \eqref{mazur_estim}, the second equality from \eqref{dadague_map}, and the third inequality from the fact that $1-\me^{-t}\leq t$ for all $t\in\er$. By our assumption, $\sum_{n=1}^\infty\sqrt{t_n}<\infty$.

On the other hand,
\begin{align}\label{l2_emb_lower}
\|f(x)-f(y)\|^p&=\sum_{n=1}^\infty\|f_n(x)-f_n(y)\|^p\\
\nonumber&=\sum_{n=1}^\infty\|M_{2,p}(\varphi_{t_n}(x))-M_{2,p}(\varphi_{t_n}(y))\|^p\\
\nonumber&\geq C^p\sum_{n=1}^\infty\|\varphi_{t_n}(x)-\varphi_{t_n}(y)\|^2\\
\nonumber&=2C^p\sum_{n=1}^\infty\left(1-e^{-t_n\|x-y\|^2}\right),
\end{align}
where the inequality follows from \eqref{mazur_estim}.

Define functions $\rho_1,\rho_2$ on $[0,\infty)$ by $$\rho_1(s)=2^\frac{1}{p}C\left(\sum_{n=1}^\infty\left(1-e^{-t_ns^2}\right)\right)^\frac{1}{p}$$ and $$\rho_2(s)=\left(\frac{2\sqrt{2}}{p}\sum_{n=1}^\infty\sqrt{t_n}\right)s.$$ Then, by \eqref{l2_emb_upper} and \eqref{l2_emb_lower}, for all $x,y\in\ell_2$ we have $$\rho_1(\|x-y\|)\leq\|f(x)-f(y)\|\leq\rho_2(\|x-y\|).$$ 

Clearly both $\rho_1,\rho_2$ are nondecreasing. By Lebesgue's monotone convergence theorem, $\rho_1(s)\to\infty$ as $s\to\infty$, and therefore $f$ is a coarse embedding. Since $\rho_2(s)\to0$ as $s\to0_+$, and $$\rho_1(s)\geq2^\frac{1}{p}C\left(1-e^{-t_1s^2}\right)^\frac{1}{p}>0$$ for every $s>0$, we see that $f$ is also a uniform embedding.
\end{proof}

\section{Main Results}\label{main_results}

Let us start with a sufficient condition for the strong uniform embeddability of Orlicz sequence spaces into $\ell_p$-spaces. The proof of the following proposition is based on a construction due to Albiac \cite[proof of Proposition 4.1(ii)]{albiac}.

\begin{pro}\label{emb_orlicz_lp}
Let $M$ be an Orlicz function with $\beta_M<\infty$ and let $p>\beta_M$. Then $h_M$ strongly uniformly embeds into $\ell_p$. 
\end{pro}
\begin{proof}
We may clearly suppose that $M(1)=1$. Fix arbitrary $q$ such that $\beta_M<q<p$. Then there is $C>0$ such that
\begin{equation}\label{first_bound_beta_M}
\frac{M(\lambda t)}{M(\lambda)t^q}\geq C\quad\text{for all $\lambda,t\in(0,1]$}.
\end{equation}
We may suppose without loss of generality that 
\begin{equation}\label{bound_beta_M}
\frac{M(\lambda t)}{M(\lambda)t^q}\geq C\quad\text{for all $\lambda>0$ and $t\in(0,1]$}.
\end{equation}
Indeed, if \eqref{first_bound_beta_M} holds, then in particular $M(t)\geq Ct^q$ for every $0<t\leq1$. We may clearly suppose that $Ct^q\leq M(t)\leq Dt^q$ for some $D\geq1$ and for every $t>1$. Then if $\lambda>1$ and $t\in(0,1]$, we have $M(\lambda t)\geq C(\lambda t)^q=C\lambda^qt^q\geq\frac{C}{D}M(\lambda)t^q$. Since $\frac{C}{D}\leq C$, we may take as $C$ in \eqref{bound_beta_M} the number $\frac{C}{D}$.

We will proceed in two steps.

\noindent{\bfseries Step 1:} We will construct functions $f_{n,k}:\er\to[0,\infty)$, $n,k\in\zet$, such that for certain constant $A\geq 1$ and for all $s,t\in\er$ we have 
\begin{equation}\label{odhad_fnk}
M(|s-t|)\leq\sum_{n,k=-\infty}^\infty|f_{n,k}(s)-f_{n,k}(t)|^p\leq A\,M(|s-t|).
\end{equation}

Suppose that $n\in\zet$. Let $a_n=2^{n+2}M\left(\frac{1}{2^{n+1}}\right)^{\frac{1}{p}}$ and define
$$f_n(t)=\begin{cases}
a_nt&\text{if }t\in\left[0,\frac{1}{2^n}\right],\\
-a_n\left(t-\frac{1}{2^{n-1}}\right)&\text{if }t\in\left(\frac{1}{2^n},\frac{1}{2^{n-1}}\right],\\
0&\text{otherwise.}
\end{cases}$$
For $k\in\zet$, define the translation of $f_n$ by $$f_{n,k}(t)=f_n\left(t-\frac{k-1}{2^{n+1}}\right),\ t\in\er.$$ Note that for all $n,k\in\zet$ the estimate $0\leq f_{n,k}\leq a_n\frac{1}{2^n}$ holds, the Lipschitz constant of $f_{n,k}$ is $a_n$, and the support of $f_{n,k}$ is $\left[\frac{k-1}{2^{n+1}},\frac{k-1}{2^{n+1}}+\frac{1}{2^{n-1}}\right]$.

For the upper estimate in \eqref{odhad_fnk}, let $s,t\in\er$, $s\neq t$, and let $N\in\zet$ be such that $\frac{1}{2^{N+1}}<|s-t|\leq\frac{1}{2^N}$.

If $n>N$ and $k\in\zet$, then
\begin{align}\label{n>N}
|f_{n,k}(s)-f_{n,k}(t)|^p&\leq a_n^p\frac{1}{2^{np}}=4^pM\left(\frac{1}{2^{n+1}}\right)=4^pM\left(\frac{1}{2^{n-N}}\frac{1}{2^{N+1}}\right)\\
\nonumber&\leq4^p\frac{1}{2^{n-N}}M\left(\frac{1}{2^{N+1}}\right)\leq4^p\frac{1}{2^{n-N}}M(|s-t|)
\end{align}
(the first inequality follows from the fact that $0\leq f_{n,k}\leq a_n\frac{1}{2^n}$, while the second one from the convexity of $M$ and the fact that $M(0)=0$).

If $n\leq N$ and $k\in\zet$, then
\begin{align}\label{n<N}
|f_{n,k}(s)-f&_{n,k}(t)|^p\leq a_n^p|s-t|^p\leq2^{p(n+2)}M\left(\frac{1}{2^{n+1}}\right)\frac{1}{2^{pN}}\\
\nonumber&=4^p\frac{1}{2^{p(N-n)}}M\left(\frac{1}{2^{n+1}}\right)=4^p\left(\frac{1}{2^{\frac{p}{q}(N-n)}}\right)^qM\left(\frac{1}{2^{n+1}}\right)\\
\nonumber&\leq\frac{4^p}{C}M\left(\frac{1}{2^{\frac{p}{q}(N-n)}}\frac{1}{2^{n+1}}\right)=\frac{4^p}{C}M\left(\frac{1}{2^{\left(\frac{p}{q}-1\right)(N-n)}}\frac{1}{2^{N+1}}\right)\\
\nonumber&\leq\frac{4^p}{C}\frac{1}{2^{\left(\frac{p}{q}-1\right)(N-n)}}M\left(\frac{1}{2^{N+1}}\right)\\
\nonumber&\leq\frac{4^p}{C}\frac{1}{2^{\left(\frac{p}{q}-1\right)(N-n)}}M(|s-t|)
\end{align}
(the first inequality follows from the fact that the Lipschitz constant of $f_{n,k}$ is $a_n$, the third one from \eqref{bound_beta_M}, and the fourth one from the convexity of $M$ and the fact that $M(0)=0$).

Note that the estimates \eqref{n>N} and \eqref{n<N} do not depend on $k$. For $n\in\zet$, denote $S_n=\{k\in\zet:f_{n,k}(s)>0\text{ or }f_{n,k}(t)>0\}$. Clearly the cardinality of $S_n$ is at most 8. Hence, using \eqref{n>N} and \eqref{n<N},
\begin{align*}
\sum_{n,k=-\infty}^\infty|&f_{n,k}(s)-f_{n,k}(t)|^p\\
&=\sum_{n>N}\sum_{k\in S_n}|f_{n,k}(s)-f_{n,k}(t)|^p+\sum_{n\leq N}\sum_{k\in S_n}|f_{n,k}(s)-f_{n,k}(t)|^p\\
&\leq8\cdot4^p\left(\sum_{n>N}\frac{1}{2^{n-N}}+\frac{1}{C}\sum_{n\leq N}\frac{1}{2^{\left(\frac{p}{q}-1\right)(N-n)}}\right)M(|s-t|)\\
&=8\cdot4^p\left(1+\frac{1}{C}\frac{1}{1-2^{1-\frac{p}{q}}}\right)M(|s-t|).
\end{align*}
So we may take $$A=8\cdot4^p\left(1+\frac{1}{C}\frac{1}{1-2^{1-\frac{p}{q}}}\right).$$

For the lower estimate in \eqref{odhad_fnk}, suppose that $s,t\in\er$, $s<t$, and let $N\in\zet$ now satisfy $\frac{1}{2^{N+2}}<|s-t|\leq\frac{1}{2^{N+1}}$. Let $K$ be the largest $k\in\zet$ such that $s$ belongs to the support of $f_{N,k}$. Then $s\in\left[\frac{K-1}{2^{N+1}},\frac{K-1}{2^{N+1}}+\frac{1}{2^{N+1}}\right)$ and $t\in\left[\frac{K-1}{2^{N+1}},\frac{K-1}{2^{N+1}}+\frac{1}{2^N}\right)$. Hence 
\begin{align*}
|f_{N,K}(s)-f_{N,K}(t)|^p&=a_N^p|s-t|^p\geq2^{p(N+2)}M\left(\frac{1}{2^{N+1}}\right)\frac{1}{2^{p(N+2)}}\\
&=M\left(\frac{1}{2^{N+1}}\right)\geq M(|s-t|),
\end{align*}
and therefore $$\sum_{n,k=-\infty}^\infty|f_{n,k}(s)-f_{n,k}(t)|^p\geq|f_{N,K}(s)-f_{N,K}(t)|^p\geq M(|s-t|).$$

\noindent{\bfseries Step 2:} 
Define $f:h_M\to\ell_p(\en\times\zet\times\zet)$ by $$f(x)=(f_{n,k}(x_i)-f_{n,k}(0))_{(i,n,k)\in\en\times\zet\times\zet},$$ where $x=(x_i)_{i=1}^\infty$ (the fact that $f(x)\in\ell_p(\en\times\zet\times\zet)$ for every $x\in h_M$ will follow from the estimates below). Let us show that $f$ is a strong uniform embedding.

Let $x=(x_i)_{i=1}^\infty,y=(y_i)_{i=1}^\infty\in h_M$. By \eqref{odhad_fnk}, for each $i\in\en$ we have $$M(|x_i-y_i|)\leq\sum_{n,k=-\infty}^\infty|f_{n,k}(x_i)-f_{n,k}(y_i)|^p\leq A\,M(|x_i-y_i|),$$ and therefore $$\sum_{i=1}^\infty M(|x_i-y_i|)\leq\|f(x)-f(y)\|^p\leq A\sum_{i=1}^\infty M(|x_i-y_i|).$$ By Lemma \ref{sum_vs_norm}, if $\|x-y\|\leq1$, then $$\|f(x)-f(y)\|^p\leq A\sum_{i=1}^\infty M(|x_i-y_i|)\leq A\|x-y\|,$$ and if $\|x-y\|\geq1$, then $$\|f(x)-f(y)\|^p\geq\sum_{i=1}^\infty M(|x_i-y_i|)\geq\|x-y\|.$$ If $\|x-y\|>1$, then, by \eqref{bound_beta_M}, for every $i\in\en$ we have $$M\left(\frac{|x_i-y_i|}{\|x-y\|}\right)\geq CM(|x_i-y_i|)\frac{1}{\|x-y\|^q},$$ and therefore, using also \eqref{sum_eq1}, we obtain
\begin{align*}
\|f(x)-f(y)\|^p&\leq A\sum_{i=1}^\infty M(|x_i-y_i|)\leq\frac{A}{C}\sum_{i=1}^\infty M\left(\frac{|x_i-y_i|}{\|x-y\|}\right)\|x-y\|^q\\
&=\frac{A}{C}\|x-y\|^q.
\end{align*}
If $\|x-y\|<1$, then similarly
\begin{align*}
\|f(x)-f(y)\|^p&\geq\sum_{i=1}^\infty M(|x_i-y_i|)\geq C\sum_{i=1}^\infty M\left(\frac{|x_i-y_i|}{\|x-y\|}\right)\|x-y\|^q\\
&=C\|x-y\|^q.
\end{align*}

Now define 
$$\rho_1(t)=\begin{cases}
C^\frac{1}{p}t^\frac{q}{p}&\text{if }t\in[0,1),\\
t^{\frac{1}{p}}&\text{if }t\geq1,
\end{cases}$$
and
$$\rho_2(t)=\begin{cases}
A^{\frac{1}{p}}t^{\frac{1}{p}}&\text{if }t\in[0,1],\\
\left(\frac{A}{C}\right)^{\frac{1}{p}}t^\frac{q}{p}&\text{if }t>1.
\end{cases}$$
Then $\rho_1,\rho_2$ are nondecreasing (since $C\leq1$), $\lim_{t\to\infty}\rho_1(t)=\infty$, and for all $x,y\in h_M$ we have $$\rho_1(\|x-y\|)\leq\|f(x)-f(y)\|\leq\rho_2(\|x-y\|).$$ Hence $f$ is a coarse embedding, and clearly it is also a uniform embedding. Since $\ell_p(\en\times\zet\times\zet)$ is isometric to $\ell_p$, we have obtained a strong uniform embedding of $h_M$ into $\ell_p$.
\end{proof}

We are now ready to give a sufficient condition for the strong uniform embeddability between Orlicz sequence spaces. Recall that if $M,N$ are metric spaces and $f:M\to N$ is a mapping, then $f$ is called a \emph{Lipschitz embedding} provided $f$ is injective and both $f$ and $f^{-1}:f(M)\to M$ are Lipschitz mappings. Clearly if $f$ is a Lipschitz embedding, then $f$ is a strong uniform embedding. 

\begin{thm}\label{orlicz_positive}
Let $M,N$ be Orlicz functions. If $\beta_M<\beta_N$ or $\beta_N\leq\beta_M<2$ or $\beta_M=\beta_N=\infty$, then $h_M$ strongly uniformly embeds into $h_N$.
\end{thm}
\begin{proof}
If $\beta_N=\infty$, then $c_0$ linearly embeds into $h_N$ by Theorem \ref{lt_embeddings}, and since every separable metric space Lipschitz embeds into $c_0$ by \cite{aharoni}, we conclude that any $h_M$ even Lipschitz embeds into $h_N$. So suppose that $\beta_N<\infty$.

If $\beta_M<\beta_N$, then $h_M$ strongly uniformly embeds into $\ell_{\beta_N}$ by Proposition \ref{emb_orlicz_lp}, and $\ell_{\beta_N}$ linearly embeds into $h_N$ by Theorem \ref{lt_embeddings}. Hence $h_M$ strongly uniformly embeds into $h_N$.

If $\beta_N\leq\beta_M<2$, then $h_M$ strongly uniformly embeds into $\ell_2$ by Proposition~\ref{emb_orlicz_lp}. By Theorem \ref{l2_emb}, $\ell_2$ strongly uniformly embeds into $\ell_{\beta_N}$, which in turn linearly embeds into $h_N$ by Theorem \ref{lt_embeddings}, and therefore $h_M$ strongly uniformly embeds into~$h_N$.
\end{proof}

To give a condition ensuring the nonexistence of a coarse or uniform embedding between two Orlicz sequence spaces, we will use the following result due to Mendel and Naor. Recall that if $X$ is a Banach space, then we define $q_X=\inf\left\{q\geq2:X\text{ has cotype }q\right\}$. 

\begin{thm}[\text{\cite[Theorems 1.9 and 1.11]{mendelnaor}}]\label{menaor_cotype}
Let $Y$ be a Banach space with nontrivial type and let $X$ be a Banach space which coarsely or uniformly embeds into $Y$. Then $q_X\leq q_Y$.
\end{thm}

\begin{thm}\label{orlicz_negative}
Let $M,N$ be Orlicz functions. If $\beta_M>2$ and $\beta_N<\beta_M$, then $h_M$ does not coarsely or uniformly embed into $h_N$.
\end{thm}
\begin{proof}
Suppose for a contradiction that $h_M$ coarsely or uniformly embeds into $h_N$. Pick any $p\in(\beta_N,\beta_M)$. Then $h_N$ strongly uniformly embeds into $\ell_p$ by Proposition~\ref{emb_orlicz_lp}, and therefore $h_M$ coarsely or uniformly embeds into $\ell_p$. But $\ell_p$ has nontrivial type (since $p>1$) and, by \eqref{cotype_orlicz}, $$q_{h_M}=\max(2,\beta_M)>\max(2,p)=q_{\ell_p},$$ which contradicts Theorem \ref{menaor_cotype}.
\end{proof}

Theorems \ref{orlicz_positive} and \ref{orlicz_negative} give an almost complete classification of the coarse (uniform) embeddability between Orlicz sequence spaces. In the remaining cases, when $\beta_N\leq\beta_M=2$ or $2<\beta_M=\beta_N<\infty$, the situation is more complicated.

Let us now investigate the case when $\beta_N\leq\beta_M=2$. We will show that in this case the coarse (uniform) embeddability of $h_M$ into $h_N$ is not determined by the values of $\beta_M$ and $\beta_N$. More precisely, for any $1\leq p\leq2$ we can find Orlicz functions $M_1,N_1,M_2,N_2$ such that $\beta_{M_1}=\beta_{M_2}=2$ and $\beta_{N_1}=\beta_{N_2}=p$, and such that $h_{M_1}$ coarsely (uniformly) embeds into $h_{N_1}$ and $h_{M_2}$ does not coarsely (uniformly) embed into $h_{N_2}$. Of course, by Theorem~\ref{l2_emb}, $\ell_2$ strongly uniformly embeds into $\ell_p$, providing thus examples of $M_1$ and $N_1$. Let us give examples of $M_2$ and $N_2$. 

We will use the following theorem due to Johnson and Randrianarivony. 

\begin{thm}[\text{\cite[Theorem 1]{johnsonrandria}}]\label{jr}
Let $X$ be a Banach space with a normalized symmetric basis $(e_n)_{n=1}^\infty$ such that $$\liminf_{n\to\infty}\frac{1}{n^\frac{1}{2}}\left\|\sum_{i=1}^ne_i\right\|=0.$$ Then $X$ does not coarsely or uniformly embed into a Hilbert space.
\end{thm}

This theorem was originally stated only for coarse embeddability; the statement about uniform embeddability follows from a result of Randrianarivony \cite[a paragraph before Theorem 1]{randriana}, who proved that a Banach space coarsely embeds into a Hilbert space if and only if it uniformly embeds into a Hilbert space.

\begin{pro}\label{nonembeddingtol2}
Let $M$ be an Orlicz function such that $$\lim_{t\to0_+}\frac{M(t)}{t^2}=0.$$ Then $h_M$ does not coarsely or uniformly embed into $\ell_2$.
\end{pro}
\begin{proof}
We may suppose without loss of generality that $M(1)=1$. Then the sequence of canonical vectors $(e_n)_{n=1}^\infty$ forms a normalized symmetric basis of $h_M$. Furthermore,
\begin{align*}
\left\|\sum_{i=1}^ne_i\right\|&=\inf\left\{\rho>0:\sum_{i=1}^nM\left(\frac{1}{\rho}\right)\leq1\right\}=\inf\left\{\rho>0:M\left(\frac{1}{\rho}\right)\leq\frac{1}{n}\right\}\\
&=\inf\left\{\rho>0:\frac{1}{\rho}\leq M^{-1}\left(\frac{1}{n}\right)\right\}=\frac{1}{M^{-1}\left(\frac{1}{n}\right)},
\end{align*}
and therefore $$\frac{1}{n^\frac{1}{2}}\left\|\sum_{i=1}^ne_i\right\|=\frac{1}{n^\frac{1}{2}M^{-1}\left(\frac{1}{n}\right)}.$$

Let $t_n=M^{-1}\left(\frac{1}{n}\right)$. Then $t_n\to0$ and $M(t_n)=\frac{1}{n}$, and therefore
\begin{align*}
\frac{1}{n^\frac{1}{2}M^{-1}\left(\frac{1}{n}\right)}=\frac{M(t_n)^\frac{1}{2}}{t_n}\xrightarrow{n\to\infty}0,
\end{align*}
since $\lim_{t\to0_+}\frac{M(t)}{t^2}=0$. 

Hence $$\liminf_{n\to\infty}\frac{1}{n^\frac{1}{2}}\left\|\sum_{i=1}^ne_i\right\|=0,$$ and therefore, by Theorem \ref{jr}, the space $h_M$ does not coarsely or uniformly embed into $\ell_2$.
\end{proof}

\begin{exa}\label{ex_no_emb}
There exists an Orlicz function $M$ such that $\alpha_M=\beta_M=2$ and $h_M$ does not coarsely or uniformly embed into $\ell_p$ for any $1\leq p\leq2$.
\end{exa}
\begin{proof}
Let $$f(t)=\frac{t^2}{1-\log t},\ t\in (0,\mathrm{e}).$$ Then using simple calculus we see that $f$ is a continuous convex function, $f(t)>0$ for each $t\in(0,\mathrm{e})$ and $\lim_{t\to 0_+}f(t)=0$. Clearly there exists an Orlicz function $M$ such that $M(t)=f(t)$ for every $t\in(0,1]$.

Let us show that $\alpha_M=\beta_M=2$. Let $q\leq2$ and $\lambda,t\in(0,1]$. Then $$\frac{M(\lambda t)}{M(\lambda)t^q}=\frac{\frac{(\lambda t)^2}{1-\log(\lambda t)}}{\frac{\lambda^2}{1-\log\lambda}t^q}=t^{2-q}\frac{{1-\log\lambda}}{1-\log(\lambda t)}\leq t^{2-q}\leq1$$ (the first inequality follows from the fact that $s\mapsto1-\log s$ is decreasing), and therefore $\alpha_M\geq2$.

Let $q>2$. If $\lambda,t\in(0,1]$, then $$\frac{M(\lambda t)}{M(\lambda)t^q}=t^{2-q}\frac{{1-\log\lambda}}{1-\log(\lambda t)}=t^{2-q}\frac{{1-\log\lambda}}{1-\log\lambda-\log t}=\frac{t^{2-q}}{1+\frac{-\log t}{1-\log\lambda}}
\geq\frac{t^{2-q}}{1-\log t},$$ where the inequality holds since $-\log t\geq0$ and $1-\log\lambda\geq1$. Now if we define $g(s)=\frac{s^{2-q}}{1-\log s}$, $s\in(0,1]$, then $\lim_{s\to0_+}g(s)=\infty$, $g(1)=1$ and $g(s)>0$ for each $s\in(0,1]$. It follows that there is $C>0$ such that $g(s)\geq C$ for each $s\in(0,1]$. Hence $$\frac{M(\lambda t)}{M(\lambda)t^q}\geq C$$ for all $\lambda,t\in(0,1]$. This implies that $\beta_M\leq2$.

Finally, if $t\in(0,1]$, then $$\frac{M(t)}{t^2}=\frac{\frac{t^2}{1-\log t}}{t^2}=\frac{1}{1-\log t}\xrightarrow{t\to0_+}0.$$ Hence, by Proposition \ref{nonembeddingtol2}, $h_M$ does not coarsely or uniformly embed into $\ell_2$. Let $1\leq p<2$. Since $\ell_p$ strongly uniformly embeds into $\ell_2$ by Theorem \ref{lp_embeddings}, it follows that $h_M$ does not coarsely or uniformly embed into $\ell_p$. 
\end{proof}

The last remaining case is when $2<\beta_M=\beta_N<\infty$. In this case, we can of course always have the coarse (uniform) embeddability (since any Banach space strongly uniformly embeds into itself). However, we do not know whether there exist Orlicz functions $M,N$ satisfying $2<\beta_M=\beta_N<\infty$, such that $h_M$ does not coarsely (uniformly) embed into $h_N$.

Let us conclude with a brief summary of the results. Let $M,N$ be Orlicz functions.

\begin{enumerate}\renewcommand{\labelenumi}{\rm(\arabic{enumi})}
\item If $\beta_M<\beta_N$ or $\beta_N\leq\beta_M<2$ or $\beta_M=\beta_N=\infty$, then $h_M$ strongly uniformly embeds into $h_N$.

\item If $\beta_M>2$ and $\beta_N<\beta_M$, then $h_M$ does not coarsely or uniformly embed into $h_N$.

\item If $\beta_N\leq\beta_M=2$, then the coarse (uniform) embeddability of $h_M$ into $h_N$ is not determined by the values of $\beta_M$ and $\beta_N$.

\item If $2<\beta_M=\beta_N<\infty$, then the question of the coarse (uniform) embeddability of $h_M$ into $h_N$ is open.
\end{enumerate}

\begin{acknow}
The author would like to thank Gilles Lancien for many inspiring discussions and helpful comments on this paper. He would also like to thank Tony Proch\'{a}zka for several valuable conversations. The author is grateful to F. Baudier for pointing out to him Albiac's construction of embeddings between $\ell_p$-spaces, and to both F. Albiac and F. Baudier for providing him with their manuscript \cite{albiacbaudier}.

This work was completed during the author's stay at Universit\'{e} de Franche-Comt\'{e} in Besan\c{c}on. The author wishes to thank the staff for their hospitality.
\end{acknow}

\end{document}